\documentclass{amsart}
\usepackage{todonotes}
\usepackage{xypic}
\usepackage{graphicx, amsmath, amssymb, amsthm}
\usepackage{hyperref} 
\usepackage[mathscr]{euscript}

\newcommand{\support}[1]{The author was supported by {#1}.}

\newcommand{\NSFThree}{NSF Grant DMS-1509652}

\newcommand{\MAHAddress}{University of California Los Angeles, Los Angeles, CA 90095}
\newcommand{\MAHemail}{\tt{mikehill@math.ucla.edu}}

\newcommand{\R}{{\mathbb R}}

\newcommand{\MUR}{MU_{\R}}

\DeclareMathOperator{\Map}{Map}

\DeclareMathOperator{\Spec}{Spec}

\DeclareMathOperator{\Supp}{Supp}
\DeclareMathOperator{\coSupp}{V}
\newcommand{\ChromPrime}{\mathscr{P}}

\newcommand{\m}[1]{{\protect\underline{#1}}}

\newcommand{\mSet}{\m{\Set}}

\newcommand{\cc}[1]{\mathcal #1}

\newcommand{\cO}{\cc{O}}
\newcommand{\ccZ}{\cc{Z}}

\newcommand{\Sp}{\mathcal Sp}

\newcommand{\Set}{\mathcal Set}

\newcommand{\Comm}{\mathcal Comm}
\newcommand{\Ninfty}{N_\infty}

\mathchardef\mhyphen=45

\numberwithin{equation}{section}

\newtheorem{theorem}{Theorem}[section]
\newtheorem{lemma}[theorem]{Lemma}
\newtheorem{corollary}[theorem]{Corollary}

\newtheorem{proposition}[theorem]{Proposition}

\newtheorem*{theorem*}{Theorem}
\newtheorem*{proposition*}{Proposition}

\theoremstyle{remark}
\newtheorem{remark}[theorem]{Remark}

\theoremstyle{definition}

\newtheorem{definition}[theorem]{Definition}

\newcommand{\defemph}[1]{\textbf{#1}}

\begin{document}
 
\title[Equivariant Localizations]{Equivariant Chromatic Localizations and Commutativity}

\author{Michael~A.~Hill}
\address{\MAHAddress}
\email{\MAHemail}
\thanks{\support{\NSFThree}}

\begin{abstract}
In this paper, we study the extent to which Bousfield and finite localizations relative to a thick subcategory of equivariant finite spectra preserve various kinds of highly structured multiplications. Along the way, we describe some basic, useful results for analyzing categories of acyclics in equivariant spectra, and we show that Bousfield localization with respect to an ordinary spectrum (viewed as an equivariant spectrum with trivial action) always preserves equivariant commutative ring spectra.
\end{abstract}


\maketitle

\section{Introduction}
Bousfield localization is a fundamental tool in modern algebraic topology. The ability to focus on pieces of the stable homotopy category allows in many cases for more conceptual or algebraic descriptions and computations. Just as in ordinary algebra, classical Bousfield localization is always a lax monoidal functor, preserving commutative ring objects and allowing one to talk about localizations in categories of modules or algebras.  

Bousfield localization plays an equally important role in equivariant homotopy, but here, the functors need not preserve commutative ring spectra. This was originally shown by McClure for the Greenlees-May Tate spectrum \cite{McClure}, and in work with Hopkins, we showed sufficient conditions for when a general Bousfield localization preserves equivariant commutative ring spectra \cite{HHLocalization}. Moreover, we showed that Bousfield localization always preserves algebras over a trivial $E_\infty$ operad (so an $E_\infty$ operad viewed as a $G$-equivariant operad by endowing it with a trivial action), and in work with Blumberg, we verified that this is sufficient to have a good symmetric monoidal catgegory of modules \cite{BHNinftyCat}. Thus essentially all of the desired classical properties hold.

For equivariant homotopy, however, we can ask for more. If $R$ is a genuine equivariant commutative ring spectrum, then the category of $R$-modules has a natural $G$-symmetric monoidal enhancement. 
More generally, if $R$ is an algebra over a linear isometries operad, then the category of $R$-modules inherits those norms which the linear isometries operad parameterizes. It is therefore a natural question to see if a particular Bousfield localization preserves these richer structures. 

In this paper, we study Bousfield and finite localizations for equivariant chromatic localizations. Using Balmer's notion of the spectrum of a tensor triangulated category, Balmer and Sanders determined the prime spectrum of the category of $G$-spectra for finite groups $G$. Coupled with the natural Balmer-Zariski topology on this spectrum, this provides a complete classification of the thick subcategories of finite $G$-spectra. Associated to any such thick subcategory are corresponding Bousfield and finite localizations, and these are our primary focus. In particular, we prove in Theorem~\ref{thm:ChromaticBousfieldComm} below sufficient conditions for these localizations to preserve various operadic algebras.

Along the way, we provide several tools which are helpful in analyzing equivariant Bousfield localizations. The role of geometric fixed points here cannot be overstated, as it provides elegant (and surprisingly checkable) reformulations of what it means for a $G$-spectrum $Z$ to be acyclic. This has several amusing consequences for the kinds of spectra which arose in the solution to the Kervaire invariant one problem which we could not resist including.

\subsection*{Notation and conventions}
In all that follows, $G$ will denote a fixed finite group. In general, the letters $H$, $J$ and $K$ will be reserved for subgroups of $G$. Capital letters close to $X$ in the alphabet will denote $G$-spaces, while capital letters close to $T$ will denote $G$-sets. Spectra will be often denoted with letters like $E$ (or $Z$ when the role as an acyclic is being stressed).

\subsubsection*{Category Names and assumptions}
We work in the category of genuine $G$-spectra, and all of our statements are implicitly homotopical. For concreteness, the reader is invited to use orthogonal $G$-spectra, where all the needed homotopical properties were checked in \cite[Appendix B]{HHR}. The category of genuine $G$-spectra will be denoted $\Sp^{G}$ and the category of spectra will be denoted $\Sp$. For either of these, the full subcategory of compact objects will be indicated by a subscript ``$c$''.

The category of genuine equivariant commutative ring spectra (the commutative monoids in one of the good symmetric monoidal model categories of $G$-spectra) will be denoted $\Comm^{G}$.

The category of finite $G$-sets and $G$-equivariant maps will be denoted $\Set^{G}$.

\subsubsection*{Familiar functors}
The geometric fixed points functor will be denoted $\Phi^{G}$. 

If $H\subset G$, then $i_{H}^{\ast}$ will denote the restriction functor from $G$-spectra to $H$-spectra. The functor $\Phi^{H}$ will also be used to denote the composite functor $\Phi^{H}\circ i_{H}^{\ast}$ on genuine $G$-spectra.

\subsection*{Acknowledgements}

We thank Tyler Lawson and Andrew Blumberg for careful readings of interminably many drafts of this short paper. Extra special thanks go to Paul Balmer for carefully explaining to us on several occasions the construction and properties of his spectrum and for closely reading an early draft.

\section{Equivariant commutativity}

\subsection{The norm and geometric fixed points}
One of the most important tools developed in the solution of the Kervaire invariant one problem was a homotopically meaningful norm functor
\[
N_H^G\colon\Sp^H\to\Sp^G.
\]
This is a strong symmetric monoidal left Quillen functor, and on equivariant commutative ring spectra, it participates in a Quillen adjunction as the left adjoint to the forgetful functor:
\[
N_H^G\colon \Comm^H\rightleftarrows\Comm^G\colon i_H^\ast.
\]
In particular, for any $G$-equivariant commutative ring spectrum $R$, there is a canonical map of $G$-equivariant commutative ring spectra
\[
N_H^G i_H^\ast R\to R.
\]
It is the requirement that a localization play nicely with these maps that confounds equivariant Bousfield localization. From a homotopical point of view, these are extra structure which we must control the behavior of. In our analysis of various trivial and chromatic localizations, we will also need to understand the geometric fixed points of the norm functor. Luckily, this is easily determined by a kind of generalized diagonal.

\begin{lemma}[{\cite[Proposition B.209]{HHR}, \cite[Proposition 2.19]{GoodTHH}}]\label{lem:Diagonal}
For any $K, H\subset G$ and for any $H$-spectrum $E$, the diagonal gives an equivalence of spectra
\[
\Phi^K N_H^G E\xrightarrow{\simeq} \bigwedge_{g\in K\backslash G/H} \Phi^{K^{g}\cap H} E.
\]
\end{lemma}
%

It is conceptually convenient to also include notation for several other endo-functors of the category of $G$-spectra.

\begin{definition}
For any $H\subset G$ and for any $G$-spectrum $E$, let
\[
N^{G/H}(E):= N_H^G i_H^\ast E.
\]
If 
\[
T=G/H_1\amalg\dots\amalg G/H_n
\]
is a finite $G$-set, then let
\[
N^{T}E:=\bigwedge_{i=1}^n N^{G/H_i}E.
\]
\end{definition}
\begin{remark}
Although the definition as given involves choices of orbit decompositions, one can make this coordinate free by defining the $T$-norm as the symmetric monoidal pushforward of the constant $B_GT$-shaped diagram with value $E$, where $B_GT$ is the translation category of $T$, mirroring the original discussion in \cite[Appendix A.3]{HHR}.
\end{remark}

\subsection{$\Ninfty$ operads and algebras}
The failure of equivariant Bousfield localization to preserve commutative ring spectra should be viewed as a peculiarity of the monoidal model structure: in genuine $G$-spectra, the commutative monoids are have not only a homotopy coherent commutative multiplication (an ordinary $E_\infty$ structure) but also coherent norm maps relating the value of the ring at various subgroups. Classically, this is packaged via a $G$-$E_\infty$ operad. In work with Blumberg, we generalized the notion of a $G$-$E_\infty$ operad to cover all kinds of coherently commutative multiplications with some norms on genuine $G$-spectra. We briefly review the relevant details now.

\begin{definition}[{\cite[Defintion 3.7]{BHNinfty}}]\label{def:Ninfty}
An $\Ninfty$ operad is an operad $\cO$ in $G$-spaces such that
\begin{enumerate}
\item The space $\cO_0$ is $G$-contractible,
\item the action of $\Sigma_n$ on $\cO_n$ is free, and
\item the space $\cO_n$ is a universal space for a family $\mathcal F_n(\cO)$ of subgroups of $G\times\Sigma_n$ which contains all subgroups of the form $H\times\{e\}$.
\end{enumerate}
\end{definition}

There is a purely combinatorial way to package the collection of subgroups which show up in the families for an $\Ninfty$ operad, and this is closely connected to the structure of algebras over the operad.  For this, recall that a symmetric monoidal coefficient system is a contravariant functor from the orbit category of $G$ to the category of symmetric monoidal categories and strong monoidal functors. The prototype of such a symmetric monoidal category is $\mSet$, for which
\[
\mSet(G/H):=\Set^H.
\]

\begin{definition}
An indexing system $\m{\cO}$ is a full, symmetric monoidal sub-coefficient system of $\mSet$ such that
\begin{enumerate}
\item For all orbits $G/H$, $\m{\cO}(G/H)$ is closed under finite limits, and
\item if $H/K\in \m{\cO}(G/H)$ and $T\in\m{\cO}(G/K)$, then $H\times_KT\in \m{\cO}(G/H)$.
\end{enumerate}
\end{definition}
In particular, any indexing system contains all trivial sets and is closed under passage to subobjects. 

Associated to any $\Ninfty$ operad $\cO$ is an indexing system $\m{\cO}$.

\begin{definition}[{\cite[Definition 4.3]{BHNinfty}}]
Let $\cO$ be an $\Ninfty$ operad. Let $T$ be a finite $H$-set of cardinality $n$, classified by a map $H\to\Sigma_n$, and let $\Gamma_T$ be the graph of this homomorphism. Then $T$ is in $\m{\cO}(G/H)$ if and only if $\cO_n^{\Gamma_T}\simeq \ast$. If $T\in\m{\cO}(G/H)$, then we say that {\defemph{$T$ is admissible}}.
\end{definition}

This construction gives an equivalence of categories, so we will henceforth ignore the distinction between an $\Ninfty$ operad and an indexing system.

\begin{theorem}
The assignment $\cO\mapsto\m{\cO}$ gives a fully-faithful embedding of the homotopy category of $\Ninfty$ operads into the poset of indexing systems \cite[Theorem 3.24]{BHNinfty} which is essentially surjective \cite{Rubin}, \cite[Corollary IV]{BonventrePer}, \cite[Section 4]{GutierrezWhite}.
\end{theorem}

Since the poset of indexing systems has an initial object $\m{\cO}^{tr}$ consisting of the indexing system of sets with a trivial action, there is a homotopy initial $\Ninfty$ operad. This is just an ordinary, non-equivariant $E_\infty$ operad viewed as a $G$-operad by endowing it with a trivial $G$-action. Thus any $\cO$-algebra has a canonical coherently commutative multiplication, since it is an algebra over a trivial $E_\infty$ operad. The role of the indexing system here is to parameterize the additional norms present in an $\cO$-algebra.

\begin{theorem}[{\cite[Lemma 6.6]{BHNinfty}}]
If $H/K$ is an admissible $H$-set for $\cO$, and $R$ is an $\cO$-algebra in spectra, then we have a contractible space of maps
\[
N_K^Hi_K^\ast R\to i_H^\ast R.
\]
\end{theorem}

\section{Equivariant Bousfield classes}

\subsection{Equivariant localizing subcategories} 
\begin{definition}
If $E$ is a $G$-spectrum, let $\ccZ_E^G$ denote the category of $E$-acyclics: the full subcategory of $\Sp^G$ consisting of all $Z$ such that $E\wedge Z$ is equivariantly contractible.

If we are working non-equivariantly, then the acyclics will be denoted simply $\ccZ_E$ with no superscript.
\end{definition}

Since geometric fixed points detect weak equivalences and are strong symmetric monoidal, this gives another, conceptually simpler way to understand membership in $\ccZ_E$.

\begin{proposition}\label{prop:GeomFPAcyclics}
A $G$-spectrum $Z$ is a $G$-equivariant $E$-acyclic spectrum if and only if for all $H\subset G$, $\Phi^H(Z)$ is non-equivariant $\Phi^H(E)$-acyclic spectrum: 
\[
\ccZ_E^G=\bigcap_{H\subset G} \big(\Phi^H\big)^{-1}(\ccZ_{\Phi^HE}).
\]
\end{proposition}
\begin{proof}
A genuine $G$-spectrum $E'$ is equivariantly contractible if and only if for all $H\subset G$, $\Phi^{H}(E')$ is contractible. Thus if $E$ is a fixed $G$-spectrum and $Z$ is any other $G$-spectrum, then $E\wedge Z$ is contractible if and only if for all $H\subset G$,
\[
\Phi^{H}(E\wedge Z)\simeq \Phi^{H}(E)\wedge \Phi^{H}(Z)\simeq\ast.\qedhere
\]
\end{proof}

The category of $E$-acyclics is also an equivariant subcategory in that it is closed under restriction and induction. 

\begin{proposition}\label{prop:RestrictionInduction}
For all $H\subset G$, we have natural inclusions
\[
i_H^\ast\ccZ_E^G\subset \ccZ_{i_H^\ast E}^H\text{ and }G_+\wedge_H\ccZ_{i_H^\ast E}^H\subset \ccZ_{E}^G.
\]
\end{proposition}
\begin{proof}
The first inclusion is obvious, since $i_H^\ast$ is a strong symmetric monoidal functor. For the second, let $X$ be in $\ccZ_{i_H^\ast E}$. The Frobenius relation
\[
\big(G_+\wedge_H X\big)\wedge E\simeq G_+\wedge_H(X\wedge i_H^\ast E)
\]
then shows that $G_+\wedge_H X$ is $E$-acyclic.
\end{proof}

\begin{corollary}
For all $H\subset G$ and $Z\in \ccZ_E^G$, we have
\[
G/H_+\wedge Z\in\ccZ_E^G.
\]
\end{corollary}

\subsection{Application: acyclics for Kervaire spectra}
Proposition~\ref{prop:GeomFPAcyclics} gives a way to readily determine the acyclics for the kinds of chromatic spectra which arose in the proof of the Kervaire invariant one problem. In particular, we can determine the acyclics for any of the spectra which arise as particular localizations of the norms of the Landweber-Araki Real bordism spectrum $\MUR$.

Recall from \cite[Section 5.4.2]{HHR} that if $G=C_{2^n}$, then there are classes 
\[
\bar{r}_i^G\in \pi_{(2^i-1)\rho_2} N_{C_2}^{C_{2^n}} \MUR
\]
such that 
\[
\Phi^{C_{2^n}} \Big(N_{C_2}^{C_{2^n}} \MUR \big[(N_{C_2}^{C_{2^n}} \bar{r}_i)^{-1}\big]\Big)
\]
is contractible \cite[Proposition 5.50]{HHR}.

\begin{proposition}\label{prop:UnderlyingAcyclics}
Let $G=C_{2^n}$.  Let $\bar{D}$ be any class in $\pi_{m\rho_{2^n}} N_{C_2}^{C_{2^n}} \MUR$ such that for all $C_2\subset H\subset G$, there is a $j_H$ such that $N_{C_2}^{C_{2^n}} r_{j_H}^H$ divides $\bar{D}$. Finally, let $M$ be any module over the commutative ring spectrum $N_{C_2}^{C_{2^n}} \MUR[\bar{D}^{-1}]$. Then 
\[
\ccZ_{M}^G=\big(i_{\{e\}}^{\ast}\big)^{-1} \ccZ_{i_{\{e\}}^\ast M}.
\]
\end{proposition}
\begin{proof}
The conditions ensure that all non-trivial geometric fixed points of $M$ are contractible. The result then follows from Proposition~\ref{prop:GeomFPAcyclics}.
\end{proof}

\begin{corollary}\label{cor:UnderlyingWedges}
If $M$ is a wedge of spectra $M_i$, each of which is a module over $N_{C_2}^{C_{2^n}} \MUR[\bar{D}_i^{-1}]$ for some $\bar{D}_i$ as in Proposition~\ref{prop:UnderlyingAcyclics}, then 
\[
\ccZ_{M}^G=\big(i_{\{e\}}^{\ast}\big)^{-1} \ccZ_{i_{\{e\}}^\ast M}.
\]
\end{corollary}

We deduce an immediate application to the Real Morava $K$-theories and Johnson-Wilson theories introduced by Hu-Kriz and studied extensively by Kitchloo-Wilson \cite{HuKrizANSS}, \cite{KW}.
\begin{corollary}
The equivariant Bousfield classes of $E_{\R}(n)$ and $K_\R(0)\vee\dots\vee K_\R(n)$ agree.
\end{corollary}
\begin{proof}
By Corollary~\ref{cor:UnderlyingWedges}, the Bousfield class of $K_\R(0)\vee\dots\vee K_\R(n)$ is determined by the underlying spectrum:
\[
i_{e}^\ast\big(K_\R(0)\vee\dots\vee K_\R(n)\big)\simeq K(0)\vee\dots\vee K(n).
\]
Similarly, a direct application of Proposition~\ref{prop:UnderlyingAcyclics} shows that the Bousfield class of $E_\R(n)$ is also determined by the underlying spectrum, which is $E(n)$. The result is now classical.
\end{proof}

\subsection{Localizations of \texorpdfstring{$\cO$}{O}-algebras}
Since the smash product is associative, we know that $\ccZ_E^G$ is always a tensor ideal of $\Sp^G$. In particular, it is a non-unital symmetric monoidal subcategory of $\Sp^G$. This gives another way to interpret Proposition~\ref{prop:RestrictionInduction}.

\begin{proposition}
For any $G$-spectrum $E$, the assignment 
\[
G/H\mapsto \m{\ccZ}_E(G/H):=\ccZ_{i_H^\ast E}^H\subset\Sp^H
\]
defines a non-unital symmetric monoidal sub-coefficient system of $\m{\Sp}$.
\end{proposition}

This reformulation allows us to most easily state the sufficient conditions for a localization to preserve $\cO$-algebra structures for an $\Ninfty$ operad $\cO$. The sufficiency of the following theorem was proved in \cite{HHLocalization}; the thesis of White built upon this in a more general context and also showed necessity as well \cite{WhiteThesis, White}.

\begin{theorem}[{\cite[Theorem 7.3]{HHLocalization}}, {\cite[Section 5]{White}}]\label{thm:LocalizationofOAlgebras}
Let $\cO$ be an $\Ninfty$ operad, and let $L$ be a Bousfield localization on $G$-spectra. 
If for every subgroup $H\subset G$ and for every admissible $T\in \m{\cO}(G/H)$ the category of acyclics is closed under $N^T$ then $L$ preserves $\cO$-algebras.
\end{theorem}

In this paper, we are also concerned with finite localizations (which are always known to be smashing). Here, the same result holds; the proof is identical.

\begin{theorem}\label{thm:FiniteLocalizationofOAlgebras}
Let $\cO$ be an $\Ninfty$ operad, and let $\cc{V}$ be a thick subcategory of $\Sp^G$. If for every $H\subset G$ and for every admissible $T\in\m{\cO}(G/H)$, $N^T$ restricts to an endofunctor of $\cc{V}$, then the finite localization $L_{\cc{V}}^f$ preserves $\cO$-algebras.
\end{theorem}

\begin{proposition}
If $\cc{V}$ is the thick subcategory generated by an object $E$, then the conditions of Theorem~\ref{thm:FiniteLocalizationofOAlgebras} are met provided $N^T(E)\in\cc{V}$ for all admissible $T$.
\end{proposition}
\begin{proof}
This is essentially \cite[Proposition B.170]{HHR}. In short, the norms commute with sifted colimits, and there is a formula for describing the norm of a cofiber in terms of the norms of the pieces. This reduces checking for a general object in the thick subcategory to checking for the generator.
\end{proof}

Since categories of acyclics are always non-unital symmetric monoidal subcategories, and since the equivariant thick subcategories are tensor ideals, we conclude that these localizations always preserve at least the trivial $E_\infty$-structure.

\begin{corollary}
If $L$ is any Bousfield or finite localization on $G$-spectra, then $L$ preserves trivial $E_\infty$ algebras.
\end{corollary}

Although this is less structured than we might like, it is enough structure to guarantee good, symmetric monoidal category of modules.

\begin{corollary}[{\cite[Theorem 1.1]{BHNinftyCat}}]
Let $\cO$ be an $\Ninfty$ operad. If $R$ is an $\cO$-algebra in $G$-spectra, and if $L$ is any Bousfield or finite localization, then there is a symmetric monoidal category of $L(R)$-modules.
\end{corollary}

The richer structure in a general $\cO$-algebra translates to a richer structure on the category of modules for an $\cO$-algebra $R$.

\begin{corollary}[{\cite[Section 5.2]{BHNinftyCat}}]
Let $\cO$ be an $\Ninfty$ operad of the homotopy type of the linear isometries operad for a $G$-universe $U$. If $R$ is an $\cO$-algebra in $G$-spectra, and if $L$ is any Bousfield or finite localization which preserves $\cO$-algebras, then there is an $\m{\cO}$-symmetric monoidal category of $L(R)$-modules.
\end{corollary}

\begin{remark}
There is also a very exciting $\infty$-categorical approach to the norm functors and various kinds of $G$-symmetric monoidal enhancements which arise in equivariant homotopy theory due to ongoing work of Barwick, Dotto, Glasman, Nardin, and Shah \cite{Bourbon2}. This will elegantly remove the ``linear isometries'' hypothesis, giving $\m{\cO}$-symmetric monoidal categories of  modules over any $\cO$-algebra.
\end{remark}

\subsection{Pushfowards and localization}
There is an interesting family of localizations which always preserves all of the desired multiplicative structure: localizations with respect to an ordinary ring spectrum viewed as a $G$-spectrum with a trivial action. We begin with a classical observation.

\begin{proposition}
If $R\to S$ is a map of ring spectra, then
\[
\ccZ_{R}\subset \ccZ_{S}\subset \Sp.
\]
\end{proposition}

\begin{proposition}\label{prop:BousfieldClasses}
If $E$ is a ring spectrum, then for any sub-conjugate $K\subset H$, we have
\[
\ccZ_{E^{H}}\subset \ccZ_{E^{K}}\subset\Sp.
\]
\end{proposition}

\begin{proof}
For any subgroup $H$, there is a map of equivariant ring spectra 
\[
E\to \Map_{H}(G_{+},i_{H}^{\ast}E).
\] 
Applying fixed points gives a map of ring spectra $E^{G}\to E^{H}$, giving the result.
\end{proof}

\begin{corollary}
The assignment of $G/H$ to the Bousfield class of $E^{H}$ is a contravariant functor from the orbit category to the Bousfield lattice of spectra.
\end{corollary}

The fixed point functor is a categorical right adjoint, with left adjoint the ``pushfoward''.
\begin{definition}
Let $i_\ast$ denote the push-forward functor
\[
i_\ast\colon \Sp\to \Sp^G,
\]
which is the left-adjoint to the $G$-fixed points functor.
\end{definition}

While in general it is very difficult to determine the fixed points of a smash product, when one of the factors is in the image of the pushforward, we can readily do so. 
In particular, we can simply move the fixed points past the smash product in this case.
\begin{proposition}[{\cite{HuKrizHomologyMUR}}]\label{prop:FixedPoints}
If $K$ is a spectrum and $E$ is a $G$-spectrum, then we have a natural equivalence of spectra
\[
(K\wedge E)^G\simeq K\wedge E^G.
\]
In particular,
\[
(i_\ast K)^G\simeq K\wedge (S^0)^G.
\]
\end{proposition}
\begin{corollary}
For any subgroup $H$ of $G$, for any ordinary spectrum $K$, and for any $G$-spectrum $E$, 
\[
\Phi^H(i_\ast K\wedge E)\simeq K\wedge \Phi^H(E).
\]
\end{corollary}

Combined with Proposition~\ref{prop:GeomFPAcyclics}, this gives another way to understand the acyclics for $i_\ast E$.

\begin{proposition}\label{prop:GFPAcyclicsforPushforwards}
A $G$-spectrum $Z$ is $i_\ast E$-acyclic if and only if for all $H\subset G$, $\Phi^H(Z)$ is $E$-acyclic:
\[
\ccZ_{i_\ast E}^G=\bigcap_{H\subset G} \big(\Phi^H\big)^{-1}(\ccZ_E).
\]
\end{proposition}

Proposition~\ref{prop:GFPAcyclicsforPushforwards} gives a readily checkable collection of criteria for acyclicity. In particular, since the geometric fixed points of the norm is well-understood, this quickly gives the following.

\begin{theorem}
If $E$ is any non-equivariant spectrum, then $L_{i_\ast E}$ and the associated $L_{i_\ast E}^f$ preserve $G$-equivariant commutative rings.
\end{theorem}
\begin{proof}
Theorems~\ref{thm:LocalizationofOAlgebras} and \ref{thm:FiniteLocalizationofOAlgebras} show that a sufficient condition for $L_{i_\ast E}$ or $L_{i_\ast E}^f$ to preserve commutative rings is for the category of acyclics to be closed under all norms. In other words, we need to show that if $Z\in \m{\ccZ}_{i_\ast E}(G/H)$, and $H\subset K\subset G$, then 
\[
N_H^{K} Z\in\m{\ccZ}_{i_\ast E}(G/K).
\]
$G$ plays no role in this, since $i_K^\ast (i_\ast E)=i_\ast E$, so it suffices to check this for $K=G$. Proposition~\ref{prop:GFPAcyclicsforPushforwards} then shows that it suffices to show that if $Z$ is an $H$-acylic, then for all subgroups $K$ of $G$, we have
\[
\Phi^K \big(N_H^G Z\big)\in \ccZ_{E}.
\]
Lemma~\ref{lem:Diagonal} shows that we have an equivalence
\[
\Phi^K\big(N_H^G Z\big)\simeq \bigwedge_{g\in K\backslash G/H} \Phi^{K^g\cap H} Z.
\]
Proposition~\ref{prop:GFPAcyclicsforPushforwards} then tells us again that since $Z$ is an $H$-acyclic for $i_\ast E$, we know that for all subgroups $J$ of $H$ that $\Phi^J(Z)$ is $E$-acyclic. In particular, for any $K$ and any double coset, we know $\Phi^{K^g\cap H}Z$ is $E$-acyclic, giving the result.
\end{proof}

This gives us a nice selection of equivariant chromatic types that preserve commutative ring spectra.

\begin{corollary}
If $V$ is any type $n$-spectrum that is not type $(n+1)$ (at some prime $p$ if $n>0$), then the finite chromatic localization $L_{i_\ast V}^f$ preserves commutative ring spectra.
\end{corollary}

These equivariant chromatic types are the first one considers, as they are lifted directly from the unstable information. Work of Balmer and Sanders describes all of the thick subcategories equivariantly, and we turn now to understanding which of their localizations preserve $\cO$-algebras.

\section{Equivariant thick subcategories}
\subsection{The Balmer-Sanders' Classification}
The equivariant thick subcategories of $\Sp^G_{c}$ have been classified by Balmer-Sanders using Balmer's notion of the spectrum of a tensor triangulated category \cite{BalmerSanders}, where again $\Sp^{G}_{c}$ is the full subcategory of compact objects in $G$-spectra. The Balmer spectrum of a tensor triangulated category should be thought of as an extension of the classical Zariski spectrum to a context which formally looks like the derived category of modules over a ring \cite{BalmerSpec}. Balmer describes a notion of a ``prime'' tensor triangulated ideal, and these form the points in his Zariski spectrum. Out of this space, one can recover the thick subcategories of the (essentially small) tensor triangulated category.

The heart of the Balmer-Sanders result is that the geometric fixed points functors, being a tensor triangulated functor, induces maps
\[
\Spec(\Phi^H)\colon \Spec(\Sp_{c})\to\Spec(\Sp^G_{c})
\]
for all subgroups $H\subset G$. The ``Thick subcategory theorem'' of Hopkins-Smith determines all of the prime ideals in $\Sp_{c}$.

\begin{definition}
Let $p$ be a prime. For each $0\leq m<\infty$, let $K(m,p)$ denote a Morava $K$-theory of height $m$ at the prime $p$. Finally, for $m\geq 1$, define a full subcategory of finite spectra by
\[
C_{m,p}:=\{X \mid K(m-1,p)_\ast(X)=0\}.
\]
\end{definition}

\begin{theorem}[{\cite{NilpotenceII}}]
The prime ideals in $\Sp$ are given by $C_{m,p}$ for all primes $p$ and all natural numbers $m\geq 1$. Their inclusions and intersections are as follows:
\begin{enumerate}
\item For all primes $p$ and $q$, $C_{1,p}=C_{1,q}$, which is the category of torsion finite spectra.
\item If $m<m'$, then $C_{m',p}\subset C_{m,p}$. 
\item If $p$ and $q$ are distinct primes and $m$ and $n$ are greater than one, then $C_{m,p}$ is not contained in $C_{m',q}$.
\end{enumerate}
\end{theorem}

\begin{definition}[{\cite[Definition 4.1]{BalmerSanders}}]
For each subgroup $H$ of $G$, each prime $p$, and each natural number $m$, let
\[
\ChromPrime(H,m,p):=\big(\Phi^H\big)^{-1} C_{m,p}.
\]
\end{definition}

\begin{theorem}[{\cite[Theorem 4.9]{BalmerSanders}}]
For any finite group $G$, the spectrum of $\Sp^G_{c}$ is
\[
\Spec(\Sp^G_{c})=\{\ChromPrime(H,m,p) \mid H\subset G, m\in\mathbb N, p\text{ prime}\}.
\]
\end{theorem}

Balmer-Sanders also determine the inclusions (up to a small indeterminacy), hence giving a complete classification of the thick subcategories. These are determined by the notion of ``support'' for an element.

\begin{definition}
If $X\in \Sp^G$, then the {\defemph{support}} of $X$ is the set of prime ideals not containing $X$:
\[
\Supp(X)=\{\mathfrak p | X\notin \mathfrak p\}.
\]
The {\defemph{vanishing locus}} of $X$ is the complement of the support:
\[
\coSupp(X)=\Spec(\Sp^G)-\Supp(X)=\{\mathfrak p | X\in\mathfrak p\}.
\]
\end{definition}

Thick subcategories are determined by the support. Thick subcategories are equivalent to the condition that their support be a Thomason closed subset. 

\begin{proposition}\label{prop:BalmerThick}
Let $X$ and $Y$ be finite $G$-spectra. Then the thick subcategory generated by $X$ contains $Y$ if and only if 
\[
\Supp(Y)\subset \Supp(X),
\]
or equivalently
\[
\coSupp(X)\subset\coSupp(Y).
\]
\end{proposition}

\subsection{Chromatic localizations and structured multiplications}
\subsubsection{General results for arbitrary $G$}
We now restrict attention to determining conditions which guarantee that the Bousfield and finite localizations with respect to an equivariant thick subcategory preserve $\cO$-algebras. 

\begin{lemma}\label{lem:NormsofThicks}
Let $\mathcal X$ be a thick subcategory of $\Sp^{G}_{c}$. Then the norm $N_K^H$ preserves $\mathcal X$ if for all $\ChromPrime(J,m,p)\in \coSupp(\mathcal X)$ with $J\subset H$, $\exists h\in H$ such that 
\[
\ChromPrime(K^{h}\cap J, m, p)\in\coSupp(\mathcal X).
\]
\end{lemma}

\begin{proof}
By Proposition~\ref{prop:BalmerThick}, this is equivalent to 
\[
\coSupp(i_H^* X)\subset \coSupp(N_K^H i_H^\ast X).
\]
Thus a sufficient condition is that if $K(m,p)_\ast \Phi^JX=0$, then 
\[
K(m,p)_\ast \Phi^J N_K^H i_K^\ast X=0.
\]
By Lemma~\ref{lem:Diagonal}, we have an isomorphism
\begin{multline*}
K(m,p)_\ast \big(\Phi^J N_K^H i_K^\ast X\big)\cong \\ K(m,p)_\ast \left(\bigwedge_{KhJ\in K\backslash H/J} \Phi^{J^h\cap K} X\right)\cong \bigotimes_{KhJ\in K\backslash H/J} K(m,p)_\ast(\Phi^{J^h\cap K} X)
\end{multline*}
This vanishes if and only if there is an $h\in H$ such that $\ChromPrime(J^h\cap H, m,p)$ is in the vanishing locus of $X$,  as desired.
\end{proof}

As an immediate corollary, we deduce sufficient conditions for chromatic Bousfield and finite localizations to preserve $\cO$-algebras, by Theorems~\ref{thm:LocalizationofOAlgebras} and \ref{thm:FiniteLocalizationofOAlgebras}.

\begin{theorem}\label{thm:ChromaticBousfieldComm}
The Bousfield and finite localizations with respect to a thick subcategory $\mathcal X$ preserve $\cO$ algebras if for all $H/K\in\cO(G/H)$ and for all $\ChromPrime(J,m,p)\in \coSupp(\mathcal X)$ with $J\subset H$, $\exists h\in H$ such that 
\[
\ChromPrime(K^{h}\cap J, m, p)\in\coSupp(\mathcal X).
\]
\end{theorem}

\subsubsection{Application: the $p$-local cases for $G=C_{p^n}$}
When $G$ is a cyclic group of order a power of a fixed prime $p$, then we can reduce the conditions from Theorem~\ref{thm:ChromaticBousfieldComm} to a collection of inequalities. These, together with the inclusions determined and conjectured by Balmer-Sanders, greatly cuts down the number of possible localizations which preserve commutative rings. 

For simplicity, we restrict attention to the $p$-local subcategory. Localized at primes $q\neq p$, the category of $C_{p^n}$ spectra splits, and the problem essentially becomes algebra. This splitting is in general not a splitting of $G$-symmetric monoidal categories, however, so it is not immediately clear how to couple this with the localization results above. Ongoing work of B\"ohme exactly addresses this point.

For $G=C_{p^n}$, a $p$-local thick subcategory is completely determined by a finite collection of extended natural numbers (so a natural number or infinity).

\begin{definition}
Let $\mathcal X$ be a $p$-local thick subcategory of $\Sp^{C_{p^n}}_{c}$. For each $0\leq k\leq n$, let 
\[
\ell_k=\max\Big\{\ell | \ChromPrime({C_{p^k}},\ell,p)\in\coSupp(\mathcal X)\Big\}
\]
provided this set is non-empty, and if for no $\ell$ is $\ChromPrime(C_{p^k},\ell,p)$ in the vanishing locus of $\mathcal X$, then let $\ell_k=-1$. Denote this sequence of extended integers by $\vec{\ell}(\mathcal X)$.

Conversely, given a sequence $\vec{\ell}=(\ell_0,\dots,\ell_n)$, let $\mathcal X_{\vec{\ell}}$ be defined by
\[
\coSupp(\mathcal X_{\vec{\ell}})=\bigcup_{i=0}^{n} \bigcup_{j=0}^{\ell_i} \{\ChromPrime(C_{p^i},j,p)\},
\]
where if $\ell_i=-1$, then that union is empty.
\end{definition}

By definition, if $\mathcal X$ is a $p$-local thick subcategory of $\Sp^{C_{p^n}}_{c}$, then 
\[
\mathcal X=\mathcal X_{\vec{\ell}(\mathcal X)}.
\]
However, not every sequence of integers works to give thick subcategories (as there is an implicit closure condition here). The Balmer-Sanders closure results give some loose conditions.

\begin{theorem}[]
For all $0\leq i\leq n-1$, if $\mathcal X_{\vec{\ell}}$ is a $p$-local thick subcategory, then we have inequalities
\[
\ell_i\leq \ell_{i+1}+1.
\]
\end{theorem}

Balmer-Sanders leave open if we can refine any of these further (for example, having inequalities like $\ell_0\leq \ell_2+1$), but leave this open as their ``$\log_p$-conjecture'' (\cite[Conjecture 8.7]{BalmerSanders}). We will not need anything more than these coarse bounds; any refinements would simply further restrict the kinds of thick subcategories which produce commutative ring spectra.

Applying the analysis for Theorem~\ref{thm:ChromaticBousfieldComm}, we deduce the following.

\begin{lemma}
Let $\mathcal X$ be a $p$-local, thick subcategory of $\Sp^{C_{p^n}}_{c}$. The norm $N_{C_{p^k}}^{C_{p^j}}$ preserves $\mathcal X$ if 
\[
\ell_k\geq \ell_{k+1}, \dots, \ell_{j}.
\]
\end{lemma}
\begin{proof}
Since every subgroup is normal and the subgroups are nested, this is an immediate application of Lemma~\ref{lem:NormsofThicks}.
\end{proof}

Putting this all together gives a condition for chromatic localizations to preserve commutative ring spectra.

\begin{corollary}
Let $\mathcal X$ be a $p$-local, thick subcategory of $\Sp^{C_{p^n}}_{c}$. Then the finite and Bousfield localizations nullifying $\mathcal X$ preserve commutative ring spectra if for all $0\leq i\leq n-1$, 
\[
\ell_{i+1}\leq \ell_{i}\leq \ell_{i+1}+1.
\]
\end{corollary}

The case that the sequence is constant is one we have already studied: these are the thick subcategories generated by the pushforward for a type $(n+1)$-complex. The other cases are new.

When $n=1$, there is even less to check. In this case, we have only $2$ extended integers.

\begin{corollary}
For $C_p$, the only chromatic localizations which preserve commutative ring spectra are
\[
L_{\mathcal X_{(n,n)}} \text{ and } L_{\mathcal X_{(n+1,n)}}
\]
for all $n$.
\end{corollary}

As a final corollary, the finite localizations by construction are smashing. Thus all of these results can be restated in terms of certain structured multiplications on the localized sphere spectrum.

\begin{corollary}
For $G=C_{p^n}$, let $\vec{\ell}$ denote a sequence of extended integers such that for all $0\leq i\leq n$, we have
\[
\ell_{i+1}\leq \ell_i\leq \ell_{i+1}+1,
\]
and assume that the Balmer-Sanders ``$\log_p$-conjecture'' holds. Then the chromatically localized spheres 
\[
L^f_{\mathcal X_{\vec{\ell}}}S^0
\]
are $C_p$-equivariant commutative ring spectra.
\end{corollary}

\bibliographystyle{plain}
\bibliography{Localization}

\end{document}